\documentclass[reqno]{amsart}
\usepackage{amsmath,amsfonts,amsthm}
\usepackage{color}
\usepackage{amssymb}
\usepackage{setspace}
\usepackage{xcolor}
\usepackage[backref]{hyperref}
\usepackage{graphicx}
\usepackage{subfigure}
\usepackage{enumerate}
\usepackage{datetime}
\usepackage{verbatim}
\usepackage{srcltx}
\usepackage{algorithmicx}
\usepackage{ytableau}
\usepackage[vcentermath,enableskew,noautoscale]{youngtab}
\usepackage[all]{xy}
\usepackage{enumitem}

\theoremstyle{plain}
\newtheorem{theorem}{\bf Theorem}

\newtheorem{lemma}[theorem]{\bf Lemma}
\newtheorem{corollary}[theorem]{\bf Corollary}

\theoremstyle{definition}

\newtheorem{definition}[theorem]{\bf Definition}

\newtheorem*{acknowledge} {Acknowledgement}

\newcommand{\EE}{\mathbb E}

\newcommand{\und}{\;\mbox{ and }\;}

\newcommand{\bdot}{\boldsymbol{\cdot}}

\def\({\left(}
\def\){\right)}
\def\cH{\mathcal{H}}

\begin{document}

\onehalfspace

\title[zero-sum free sequences contained in random subsets]{On zero-sum free sequences contained in random subsets of finite cyclic groups}

\author[S.~J.~Lee]{Sang June Lee}
\address{Department of Mathematics \\ Kyung Hee University \\ Seoul 02447 \\ South Korea} \email{sjlee242@khu.ac.kr, sjlee242@gmail.com}

\author[J.~S.~Oh]{Jun Seok Oh}
\address{Research Institute of Basic Sciences \\ Incheon National University \\ Incheon 22012  \\ South Korea} \email{junseok1.oh@gmail.com}

\thanks{The first author was supported by Basic Science Research Program
through the National Research Foundation of Korea (NRF) funded by
the Ministry of Education (NRF-2019R1F1A1058860).
The second author is the corresponding author. This
work was partially done while the second author was visiting Kyung Hee University in South Korea.}

\date{\today}

\begin{abstract}
Let $C_n$ be a cyclic group of order $n$. A {\it sequence} $S$ of length $\ell$ over $C_n$ is a sequence $S = a_1\bdot a_2\bdot \ldots\bdot a_{\ell}$ of $\ell$ elements in  $C_n$, where a repetition of elements is allowed and their order is disregarded. We say that $S$ is a zero-sum sequence if $\Sigma_{i=1}^{\ell} a_i = 0$ and that $S$ is a zero-sum free sequence if $S$ contains no zero-sum subsequence.

Let $R$ be a random subset of $C_n$ obtained by choosing each element in $C_n$ independently with probability $p$.
Let $N^R_{n-1-k}$ be the number of zero-sum free sequences of length $n-1-k$ in $R$. Also, let $N^R_{n-1-k,d}$ be the number of zero-sum free sequences of length $n-1-k$ having $d$ distinct elements in $R$. We obtain the expectation of $N^R_{n-1-k}$ and $N^R_{n-1-k,d}$ for $0\leq k\leq \left\lfloor n/3\right\rfloor$. We also show a concentration result on $N^R_{n-1-k}$ and $N^R_{n-1-k,d}$ when $k$ is fixed.
\end{abstract}

\maketitle

\section{Introduction}\label{sec:intro}

Let $C_n$ be a cyclic group of order $n$. A {\it sequence} $S$ of length $\ell$ over $C_n$ is a sequence $$S = a_1\bdot a_2\bdot \ldots\bdot a_{\ell}$$ of $\ell$ elements in  $C_n$, where a repetition of elements is allowed and their order is disregarded. 
We say that a sequence $S$ over $C_n$ is contained in $A\subset C_n$ if each element in $S$ is contained in $A$. For $a \in C_n$, let
\[
  \mathsf v_a (S)  =  \left| \{ i \in [1,\ell] \mid a_i = a \} \right|
\]
be the {\it multiplicity} of $a$ in $S$.
A {\it subsequence} $T$ of $S$ is a sequence over $C_n$ satisfying $\mathsf v_a (T) \le \mathsf v_a (S)$ for all $a \in C_n$.
We say that $S$ is a {\it zero-sum sequence} if $a_1 + a_2 + \ldots + a_{\ell} = 0$. A sequence is called {\it zero-sum free} if it contains no  zero-sum subsequence.

An initial study of zero-sum sequences dates back to 1961 when Erd{\H{o}}s, Ginzburg, and Ziv \cite{EGZ_1961} proved that $2n-1$ is the smallest positive integer $\ell$ such that every sequence of length $\ell$ over $C_n$ has a zero-sum subsequence of length $n$. Since that time, zero-sum sequences over a finite group have actively studied in additive combinatorics. For more details, see a survey paper by Gao and Geroldinger~\cite{Ga-Ge_2006}. Although earlier works often focused on finite abelian groups, an application to factorization theory and invariant theory pushed the object forward to non-abelian groups. The reader can refer to Geroldinger, Grynkiewicz, Zhong, and the second author~\cite{GGOZ_2019} for recent progress with respect to factorization theory and to Cziszter, Domokos, and Sz\"oll{\H{o}}si~\cite{CD_2014, CDS_2018} for connection with invariant theory.

In this paper, we focus on zero-sum free sequences over a cyclic group. Well-known problems about zero-sum free sequences over a finite group are to determine the maximum length of zero-sum free sequences, which is a combinatorial group invariant known as the {\it Davenport constant}, and to characterize the structure of zero-sum free sequences.
Observe that the maximum length of all zero-sum free sequences over $C_n$ is $n-1$. Also, we have that $S$ is a zero-sum free sequence of length $n-1$ over $C_n$ if and only if $$S = \underset{n-1}{\underbrace{g \bdot g \bdot \ldots \bdot g}}$$ for a generator $g \in C_n$. Gao~\cite{Gao_2000} proved the following result on the structure of long zero-sum free sequences over $C_n$.

\begin{theorem}[Theorem 4.3 in~\cite{Ga-Ge_2006}, Lemma 2.5 in~\cite{Gao_2000}]
\label{thm:Gao}~
Let $n \ge 2$ and $0\leq k\leq \left\lfloor \frac{n}{3} \right\rfloor$. Then $S$ is a zero-sum free sequence of length $n-1-k$ over $C_n$ if and only if
\[
  S  =   \underset{n-1-2k}{\underbrace{g \bdot \ g \bdot\ \ldots \bdot \ g}} \bdot\ (x_1 g) \bdot\  (x_2 g) \bdot\  \ldots  \bdot \ (x_k g) ,
\]
where $g$ is a generator of $C_n$ and  $x_1, x_2, \ldots, x_{k}$ are positive integers such that
\begin{equation} \label{eq:x_i}
1\leq x_1 \le x_2 \le \ldots \le x_{k} \quad \und \quad x_1 + x_2 + \ldots + x_{k} \le 2k.
\end{equation}
\end{theorem}

Theorem~\ref{thm:Gao} was generalized by Savchev and Chen~\cite{SC_2007} on the zero-sum free sequences of length at least $(n+1)/2$ over $C_n$. Theorem~\ref{thm:Gao} and the result by  Savchev and Chen were applied to the number of minimal zero-sum sequences of long length by Ponomarenko~\cite{P_2004} and Cziszter, Domokos, and Geroldinger~\cite{CDG_2016}, respectively.

Remark that Theorem 4.3 in~\cite{Ga-Ge_2006} only gives the statement in Theorem~\ref{thm:Gao} from the left-hand side to the right-hand side. The proof from the right-hand side to the left-hand side is obvious since $g$ is a generator of $C_n$ and all subsequences $T$ of $S$ satisfy
$\sigma(T)=\ell g\neq 0$ for some integer $0<\ell<n$, where $\sigma(T)$ is the sum of all elements in $T$.

In this paper, we are interested in zero-sum free sequences of a given length contained in a random subset of $C_n$.
Investigating how classical extremal results in \emph{dense} environments
transfer to \emph{sparse} settings has become a deep line of
research. For example, Roth's theorem on $3$-term
arithmetic progressions~\cite{roth53} was generalized for random subsets of
integers~\cite{KLRap3}, and there are recent generalizations about various classical extremal results by Schacht~\cite{schacht:_extrem_resul} and
Conlon and Gowers~\cite{conlon:_theor_in_spars_random}.

Let $R$ be a random subset of $C_n$ obtained by choosing each element in $C_n$ independently with probability $p$.
Let $N_{n-1-k}$ be the number of zero-sum free sequences of length $n-1-k$ over $C_n$. Also, let $N^R_{n-1-k}$ be the number of zero-sum free sequences of length $n-1-k$ in $R$. The result on the expectation of $N^R_{n-1-k}$ is as follows.
\begin{theorem} \label{thm:exp2}~
Let $n\geq 2$ and $0\leq k\leq  \left\lfloor \frac{n}{3} \right\rfloor$.  The expected number of zero-sum free sequences of length $n-1-k$ in a random subset $R$ of $C_n$ is
\begin{equation*} \label{eq:Exp100}
\EE\(N^R_{n-1-k}\)=\varphi(n) \left[ p + \sum_{d = 2}^{D} p^{d} \(\sum_{j=\frac{(d-1)d}{2}}^k q(j,d-1)\)\right],
\end{equation*}
where \begin{itemize}\item $D=\left\lfloor \frac{1+\sqrt{1+8k}}{2}\right\rfloor$, \item $\varphi(n)$ denotes the number of generators in $C_n$, and \item $q(j,d-1)$ is the number of partitions of $j$ having $d-1$ distinct parts.\end{itemize} 
\end{theorem}

The number $q(j,d-1)$ can be computed in two ways: The first way is based on its generating function (see Section~\ref{subsec:thm:exp2} for  details). Second, we provide a recursive formula for computing $X_{k,d-1}=\sum_{j=(d-1)d/2}^k q(j,d-1)$ (see Section~\ref{subsec:recursive}).

If $k$ is fixed, then we can obtain a simpler statement as follows. 
\begin{corollary}\label{coro:exp2}
If $k$ is fixed and $p=o(1)$ as $n\rightarrow \infty$, then
\begin{equation*} \label{eq:Exp3}
\EE\(N^R_{n-1-k}\)=p\varphi(n)\left(1+O_k(p)\right),
\end{equation*}
where the constant in $O_k$ depends only on $k$.
\end{corollary}

Next, we have a concentration result on $N^R_{n-1-k}$ when $k$ is fixed.

\begin{theorem}\label{thm:concentration2} Let $k$ be fixed, and let $p$ be such that
$$\frac{(\log n)^{2d} \log\log n}{n}\ll p\ll 1.$$ Then, asymptotically almost surely (a.a.s.)
$$N^R_{n-1-k}=p\varphi(n)+O_k\(p^2\varphi(n)+\sqrt{p\varphi(n)}(\log n)^d\),$$
where the constant in $O_k$ depends only on $k$.
\end{theorem}

Moreover, we have a refined result. Let $N_{n-1-k,d}$ be the number of zero-sum free sequences  of length $n-1-k$ having $d$ distinct elements over $C_n$. Also, let $N^R_{n-1-k,d}$ be the number of zero-sum free sequences of length $n-1-k$ having $d$ distinct elements contained in a random subset $R$ of $C_n$. We show a concentration result on $N^R_{n-1-k,d}$.

\begin{theorem}\label{thm:concentration}
If $0\leq k\leq \left\lfloor \frac{n}{3}\right\rfloor$ and $$ p\gg \frac{\log \log n}{n},$$ then we have that a.a.s.
$$p\varphi(n)-\omega \sqrt{p\varphi(n)}\leq N^R_{n-1-k,1}\leq p\varphi(n)+\omega \sqrt{p\varphi(n)},$$
where $\omega$ tends to $\infty$ arbitrarily slowly as $n\rightarrow \infty$.

Let $d\geq 2$.
If $k$ is fixed and $$p\gg \frac{(\log n)^2(\log\log n)^{1/d}}{n^{1/d}},$$ then we have that a.a.s.
$$N^R_{n-1-k,d}=p^d\varphi(n) \(\sum_{j=\frac{(d-1)d}{2}}^k q(j,d-1)\)+O_k\(\sqrt{p^d\varphi(n)}(\log n)^d\).$$
\end{theorem}

The organization of this paper is as follows. In Section~\ref{sec:expectation}, we consider expectations and prove Theorem~\ref{thm:exp2} and Corollary~\ref{coro:exp2}. Then, we deal with our concentration results and prove Theorems~\ref{thm:concentration2} and~\ref{thm:concentration} in Section~\ref{sec:concentration}.


\section{Expectation}\label{sec:expectation}

In this section, we prove Theorem~\ref{thm:exp2} and Corollary~\ref{coro:exp2}. Also, we provide a recursive formula to compute the important value $X_{k,d-1}=\sum_{j=(d-1)d/2}^k q(j,d-1)$ given in Theorem~\ref{thm:exp2}.

\subsection{Proofs of Theorem~\ref{thm:exp2} and Corollary~\ref{coro:exp2}}\label{subsec:thm:exp2}

It turns out that the number of distinct elements in a zero-sum free sequence plays an important role since each element in $C_n$ is contained in a random set $R$ with probability $p$. 
Recall that $N_{n-1-k,d}$ is the number of zero-sum free sequences of length $n-1-k$ having $d$ distinct elements over $C_n$, and $N^R_{n-1-k,d}$ is the number of zero-sum free sequences over $C_n$ of length $n-1-k$ having $d$ distinct elements contained in a random set $R$.

Clearly, the expectation of $N^R_{n-1-k,d}$ is
\[
\EE\(N^R_{n-1-k,d}\) =  p^d N_{n-1-k,d}.
\]
Based on Theorem~\ref{thm:Gao}, the numbers $N_{n-1-k,d}$ and $N^R_{n-1-k,d}$ are related to the number 
of $$(x_1, x_2, \ldots, x_k)$$ satisfying that $x_1, x_2, \ldots, x_k$ are positive integers such that~\eqref{eq:x_i} holds and the number of distinct $x_i \neq 1$ is $d-1$. With $x'_i:=x_i-1$, the number can be simplified as follows.
\begin{definition}\label{fact:X} Let $X_{k,d}$ be the number of $(x'_1, x'_2,\dots x'_k)$ such that $$0 \leq x'_1\leq  x'_2\leq \dots\leq x'_k, \hskip 1em
 x'_1 + x'_2 + \dots + x'_k \le k,$$
and the number of distinct positive $x'_i$ is $d$.
\end{definition}

Theorem~\ref{thm:Gao} and Definition~\ref{fact:X} give that
\begin{equation} \label{eq:N}
N_{n-1-k,d} = \varphi(n)  X_{k,d-1},
\end{equation}
where $\varphi(n)$ is the number of generators in $C_n$. Therefore, the expectation of $N^R_{n-1-k,d}$
 is
\begin{equation}\label{lem:expectation_N^R}
 \EE(N^R_{n-1-k,d})= p^d \varphi(n) X_{k,d-1}.
\end{equation}

From now on, we focus on estimating $ X_{k,d-1}$. To this end, we use the definition of a partition of an integer.
A {\it partition} of a positive integer $k$ is a non-decreasing sequence whose sum equals~$k$. A partition $\lambda$ of $k$ can be shortly expressed by
\[
  1^{r_1} \ 2^{r_2} \ \cdots \ t^{r_t}
\]
meaning that $$k = (\underset{r_1}{\underbrace{1+1+\ldots+1}}) + (\underset{r_2}{\underbrace{2+2+\ldots+2}}) + \ldots + (\underset{r_t}{\underbrace{t+t+\ldots+t}}).$$
If $\lambda$ is a partition of $k$, then we denote $\lambda \vdash k$. Let $|\lambda|=k$ if $\lambda \vdash k$.

Let $q(k,d)$ be the number of partitions of $k$ having $d$ distinct parts.
For example, all partitions of $7$ are as follows:
\begin{itemize}
\item $1^{7}, \,\, 7^{1}$,

\item $1^{5} 2^{1}, \,\, 1^{3} 2^{2}, \,\, 1^{1} 2^{3}, \,\, 1^{4} 3^{1}, \,\, 1^{1} 3^{2}, \,\, 2^{2} 3^{1}, \,\, 1^{3} 4^{1}, \,\, 3^{1} 4^{1}, \,\, 1^{2} 5^{1}, \,\, 2^{1} 5^{1}, \,\, 1^{1} 6^{1}$,

\item $1^{2} 2^{1} 3^{1}, \,\, 1^{1} 2^{1} 4^{1}$.
\end{itemize}
We have that $q(7,1) = 2$, $q(7,2) = 11$, $q(7,3) = 2$, and $q(7,d) = 0$ for  $d \ge 4$.

Recalling Definition~\ref{fact:X},  we have that $X_{k,d}$ is the same as the number of partitions $\lambda$ of at most $k$ having $d$ distinct parts. Observe that if $\lambda$ is counted for $X_{k,d}$, then
\begin{equation}\label{eq:range_d}\frac{d(d+1)}{2}\leq |\lambda|\leq k\end{equation}
because $\lambda$ contains parts with at least $1, 2, \ldots, d$. Thus, we have
\begin{equation} \label{eq:X1}
 X_{k,d} = \sum_{j=\frac{d(d+1)}{2}}^k q(j,d).
\end{equation}

Remark that the number $q(j,d)$ can be found in A116608 of the on-line encyclopedia of integer sequences (OEIS), and it can be computed from its generating function
\begin{equation*} Q(x,t)=-1+\prod_{i=1}^{\infty}\(1+\frac{tx^i}{1-x^i}\),
\end{equation*}
where $$Q(x,t)=\sum_{j,d\geq 1}q(j,d)x^jt^d.$$
There are related results on $q(j,d)$. Kim~\cite{Kim_2012} constructed a generating function with one variable for $q(j,d)$ when $d$ is fixed. Also, Goh and Schmutz~\cite{GS_1995} obtained the asymptotic distribution of the number of distinct part sizes in a random integer partition.
On the other hand, $X_{k,d}$ is not found in OEIS.

We are ready to prove Theorem~\ref{thm:exp2}.
\begin{proof}[Proof of Theorem~\ref{thm:exp2}] Trivially, the expected number of zero-sum free sequences of length $n-1-k$ with same elements  in $R$ is $$\EE\(N^R_{n-1-k,1}\)=\varphi(n)p.$$
Next, for $d\geq 2$, we infer that
\begin{equation*}
\EE\(N^R_{n-1-k,d}\)=p^d N_{n-1-k,d}\overset{\eqref{eq:N}}{=}p^d\varphi(n)X_{k,d-1}\overset{\eqref{eq:X1}}{=}p^d\varphi(n)\(\sum_{j=\frac{(d-1)d}{2}}^{k} q(j,d-1)\).
\end{equation*}

Next we consider the range of $d$. If $X_{k,d-1}$ is positive, then~\eqref{eq:range_d} gives that
$$k\geq \frac{(d-1)d}{2}.$$ Hence, let $D$ be the lagest integer $d$ satisfying $(d-1)d/2\leq k$, and then,
$$d\leq D=\left\lfloor \frac{1+\sqrt{1+8k}}{2}\right\rfloor.$$ This completes our proof of  Theorem~\ref{thm:exp2}.
\end{proof}

 Now we are ready to prove Corollary~\ref{coro:exp2} using Theorem~\ref{thm:exp2}.

\begin{proof}[Proof of Corollary~\ref{coro:exp2}]
For a fixed $k$, Theorem~\ref{thm:exp2} gives that
\begin{eqnarray*} \label{eq:Exp}
\EE\(N^R_{n-1-k}\)&=&\varphi(n) \left[ p + \sum_{d = 2}^{D} p^{d} \(\sum_{j=\frac{(d-1)d}{2}}^k q(j,d-1)\)\right] \\
&=&p\varphi(n)\left(1+\sum_{d=2}^DO_k(p^{d-1})\right) \\ &=&p\varphi(n)\left(1+O_k(p)\right),
\end{eqnarray*}
where the constant in $O_k$ depends only on $k$, which completes the proof of Corollary~\ref{coro:exp2}.
\end{proof}

\subsection{Recursive formula for $X_{k,d}$}\label{subsec:recursive}

Here, we give another way to compute the important value $$X_{k,d-1}=\sum_{j=\frac{(d-1)d}{2}}^k q(j,d-1)$$ given in Theorem~\ref{thm:exp2} using a recursive formula.

A partition of an integer can be illustrated by a {\it Young diagram} (also called a {\it Ferrers diagram}), which is a useful way to understand a partition in  combinatorics. A Young diagram corresponding to a partition $\lambda \vdash k$ is a collection of left-justified rows of $k$ boxes piled up in non-decreasing order of row lengths from parts. For example, the partition $1^{2}2^{1}3^{1}\vdash 7$ corresponds to the Young diagram $${\scriptsize\Yautoscale1\yng(3,2,1,1)}.$$

Let $Y_{b,c,d}$ be the number of partitions of at most $b$ with at most $c$ parts having $d$ distinct parts. Equivalently, $Y_{b,c,d}$ is the number of Young diagrams with at most $b$ boxes, at most $c$ rows, and $d$ distinct rows. See Figure~\ref{fig:Y} (a).
\begin{figure}[!t]
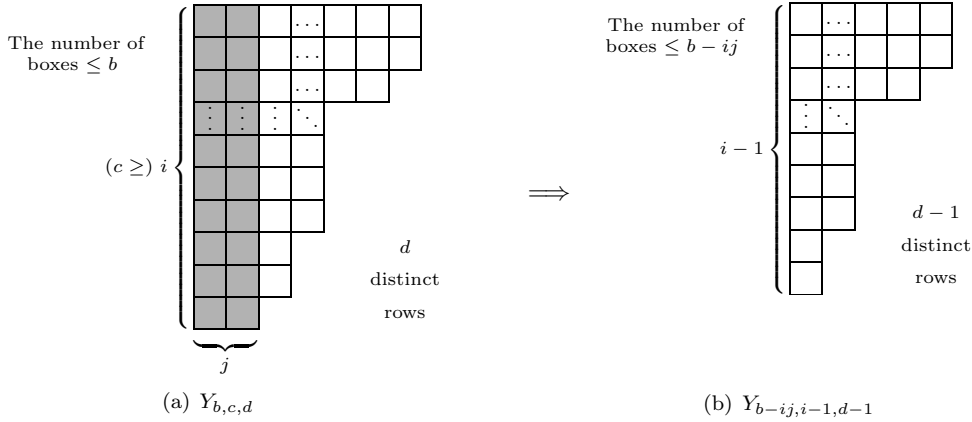

\hspace{-1cm}
\subfigure[$Y_{b,c,d}$]{{\scriptsize
\begin{tabular}{r@{}l}
\raisebox{-17ex}{$\begin{tabular}{c} The number of \\ boxes $\le b$ \vspace{3cm} \end{tabular} \hspace{-20pt} (c \ge)\ i\left\{\vphantom{\begin{array}{c}\\[39ex] \end{array}}\right.$} &
\begin{ytableau}
*(black!30)      & *(black!30)      &       & \dots &       &       &       \\
*(black!30)      & *(black!30)      &       & \dots &       &       &       \\
*(black!30)      & *(black!30)      &       & \dots &       &       & \none \\
*(black!30)\raisebox{-2.5pt}\vdots& *(black!30)\raisebox{-2.5pt}\vdots& \raisebox{-2.5pt}\vdots& \raisebox{-2.5pt}{$\ddots$}& \none& \none & \none \\
*(black!30)      & *(black!30)      &       &       & \none & \none & \none \\
*(black!30)      & *(black!30)      &       &       & \none & \none & \none \\
*(black!30)      & *(black!30)      &       &       & \none & \none & \none \\
*(black!30)      & *(black!30)      &       & \none & \none & \none & \none[d] \\
*(black!30)      & *(black!30)      &       & \none & \none & \none & \none[\textnormal{ distinct }] \\
*(black!30)      & *(black!30)      & \none & \none & \none & \none & \none[\textnormal{ rows }] \\
\end{ytableau}\\[-1.5ex]
& $\underbrace{\hspace{3em}}_{\displaystyle j}$ \vspace{5pt}
\end{tabular}}} \hspace{1cm} $\Longrightarrow$
\subfigure[$Y_{b-ij,i-1,d-1}$]{{\scriptsize
\begin{tabular}{r@{}l}
\raisebox{-15ex}{$\begin{tabular}{c} The number of \\ boxes $\le b-ij$ \vspace{3cm} \end{tabular} \hspace{-15pt} i-1\left\{\vphantom{\begin{array}{c}\\[35ex] \end{array}}\right.$} &
\begin{ytableau}
~       & \dots&        &       &            \\
~       & \dots&        &       &            \\
~       & \dots&        &       & \none      \\
\raisebox{-2.5pt}\vdots & \raisebox{-2.5pt}{$\ddots$} & \none & \none & \none & \none   \\
~       &       & \none & \none & \none      \\
~       &       & \none & \none & \none      \\
~       &       & \none & \none & \none[d-1] \\
~       & \none & \none & \none & \none[\textnormal{ distinct }] \\
~       & \none & \none & \none & \none[\textnormal{ rows }]     \\
\end{ytableau}\vspace{35pt}
\end{tabular}}}
\caption{Definition of $Y_{b,c,d}$ and the deletion process}
\label{fig:Y}
\end{figure}

Note that $Y_{b,c,d}>0$ if and only if  $b\geq \frac{d(d+1)}{2}$ and $c\geq d$, where the first inequaity follows from~\eqref{eq:range_d}.  Observe that $$X_{k,d}=Y_{k,k,d}.$$

A recursive formula for $Y_{b,c,d}$ is as follows. (Hence we have a recursive formula for $X_{k,d}$.)

\begin{lemma} We have that, for $b\geq \frac{d(d+1)}{2}$ and $c\geq d\geq 2$,
\begin{equation}\label{eq:Y(1)}Y_{b,c,d}=\sum_{i=1}^c\sum_{j=1}^{\left\lfloor \frac{b-(d-1)d/2}{i}\right\rfloor} Y_{b-ij,i-1,d-1}
\end{equation}
and, for $b\geq 1$ and $c\geq 1$,
\begin{equation}\label{eq:Y(2)}
Y_{b,c,1}=\sum_{i=1}^c\left\lfloor\frac{b}{i} \right\rfloor.
\end{equation}

\end{lemma}

\begin{proof} We first show~\eqref{eq:Y(1)}. 
We delete the gray retangle in Figure~\ref{fig:Y} from a Young diagram counted for $Y_{b,c,d}$, and then we have a Young diagram with at most $b-ij$ boxes, at most $i-1$ rows, and $d-1$ distinct rows.

We consider the ranges of $i$ and $j$. Clearly, the range of $i$ is $1\leq i\leq c$. Then the remaining Young diagram after the deletion has $d-1$ distinct rows, and hence, it has at least $(d-1)d/2$ boxs. Thus,
$$ij+\frac{(d-1)d}{2}\leq b.$$
So the range of $j$ is $$1\leq j\leq  \left\lfloor \frac{b-(d-1)d/2}{i} \right\rfloor.$$

Next, we show~\eqref{eq:Y(2)}. The number $Y_{b,c,1}$ is the same as the number of rectangles with at most $b$ boxes and at most $c$ rows. Let $i$ and $j$ be the numbers of rows and columns, respectively, of such a rectangle. Clearly, $1\leq i\leq c$. Since $ij\leq b$, we have $j\leq \left\lfloor \frac{b}{i} \right\rfloor.$
\end{proof}

\section{Concentration}\label{sec:concentration}

Recall that $N^R_{n-1-k,d}$ be the number of zero-sum free sequences of length $n-1-k$ having $d$ distinct elements in a random subset $R$. From~\eqref{lem:expectation_N^R}, recall that
\begin{equation*} \label{eq:Z_d}
\EE(N^R_{n-1-k,d}) = p^d \varphi (n) X_{k,d-1}.
\end{equation*}
From now on, we consider a concentration of $N^R_{n-1-k,d}$ and $N^R_{n-1-k}$ using a graph theoretical approach called the Kim--Vu polynomial concentration result.

\subsection{Kim--Vu polynomial concentration result}
\label{sec:Kim--Vu}
Let~$\cH=(V,E)$ be a weighted hypergraph with $V=[n]:=\{0,1,\dots,n-1\}$.  Recall that~$R$ is a
random subset of~$[n]$ obtained by selecting each $v\in [n]$
independently with probability~$p$.  Let~$\cH[R]$ be the
sub-hypergraph of~$\cH$ induced on~$R$, and we let $Z$ be the sum of weights of hyperedges in $\cH[R]$.  Kim and
Vu~\cite{Kim-Vu_2000} obtained a result that provides a concentration of~$Z$ around its mean $\EE(Z)$ with high
probability.  For more details, see Alon and Spencer \cite{AS_2000}.
To state the result, we need some definitions.

\begin{definition}
  \label{df:Kim--Vu}
  Let $\ell$ be the maximum size of hyperedges in $\cH$, and let
  $A\subset [n]$ be such that $|A|\leq \ell$.  We let
\begin{itemize}
\item $\displaystyle Z_A:=$ the sum of weights of hyperedges in $\cH[R]$ containing $A$,
\item  $\EE_A:=\EE\(Z_A\; |\; A\subset R\)$,
\item $\EE_i:=$ the maximum of $\EE_A$ for $A\subset [n] \text{ with }|A|=i$,
\item
  \begin{equation*}
    \label{eq:E',E}
  \EE':=\max_{1\leq i\leq \ell}\EE_i
    \mbox{ \hskip 1em and \hskip 1em }
    \EE^*:=\max\{\EE', \EE(Z)\}.
  \end{equation*}

\end{itemize}

\end{definition}

The concentration result by Kim and Vu~\cite{Kim-Vu_2000} is as follows.

\begin{theorem}[\textbf{Kim--Vu polynomial concentration inequality}]
  \label{thm:Kim--Vu}
  With the notation as above, we have that, for each $\lambda>1$,
  \begin{equation*}
    \label{eq:Kim--Vu}
    \Pr\left[|Z-\EE(Z)|>a_\ell\sqrt{\EE'\cdot\EE^*}\lambda^\ell\right]<2
    e^{-\lambda+2}n^{\ell-1},
  \end{equation*}
  where~$a_\ell=8^\ell (\ell!)^{1/2}$.
\end{theorem}

\subsection{Hypergraph and example}
 For a given positive integer $k$, we define the hypergraph $\cH_{n-1-k}=\cH_{n-1-k}(C_n)=([n],E)$ such that $a_1\bdot a_2\bdot \dots \bdot a_{n-1-k}$ is a zero-sum free sequence over $C_n$ if and only if the corresponding set $\{b_1, b_2, \dots, b_{\ell}\}=\{a_1, a_2, \dots, a_{n-1-k}\}$, with $1\leq \ell\leq n-1-k$, is contained in $E$. The weight of an hyperedge $\{b_1, b_2, \dots, b_{\ell}\}$ of $\cH_{n-1-k}$ is the number of zero-sum free sequences over $C_n$ consisting of $b_1, b_2, \dots, b_{\ell}$.

Then $\EE_{d,A}$ defined above is the expected number of zero-sum free sequences of length $n-1-k$ having $d$ distinct elements that contains $A\subset C_n$ and is contained in $R$ under the condition that $A\subset R$. Also, for $1\leq i\leq d$, let
\[
  \EE_{d,i} = \max \{ \EE_{d,A} \mid A \subset C_n \mbox{ with } |A| = i \}.
\]
We will estimate $\EE_{d,A}$ and $\EE_{d,i}$.

 For an easier understanding,  we give an example in $C_8$ before estimating $\EE_{d,A}$ and $\EE_{d,i}$ in a general $C_n$.
Let $C_8 = \{ 0,1,2,3,4,5,6,7 \}$ and we consider the case where $k=2$. In this case, the length of zero-sum free sequences is $n-1-k=8-1-2=5$. All generators in $C_8$ are $1,3,5,7$, and all possible $(x_1,x_2)$ in Theorem~\ref{thm:Gao} are $(1,1), (1,2), (1,3),$ and $(2,2).$ Thus, Theorem~\ref{thm:Gao} gives that all zero-sum free sequences of length $5$ over $C_8$ are
\[
  \begin{aligned}
  1\bdot 1\bdot 1\bdot 1\bdot 1 \quad \quad 3\bdot 3\bdot 3\bdot 3\bdot 3 \quad \quad 5\bdot 5\bdot 5\bdot 5\bdot 5 \quad \quad 7\bdot 7\bdot 7\bdot 7\bdot 7 \phantom{.} \\
  1\bdot 1\bdot 1\bdot 1\bdot 2  \quad\quad 3\bdot 3\bdot 3\bdot 3\bdot 6 \quad \quad 5\bdot 5\bdot 5\bdot 5\bdot 2 \quad \quad 7\bdot 7\bdot 7\bdot 7\bdot 6\phantom{.} \\
  1\bdot 1\bdot 1\bdot 1\bdot 3  \quad\quad 3\bdot 3\bdot 3\bdot 3\bdot 1 \quad \quad 5\bdot 5\bdot 5\bdot 5\bdot 7 \quad \quad 7\bdot 7\bdot 7\bdot 7\bdot 5\phantom{.} \\
  1\bdot 1\bdot 1\bdot 2\bdot 2  \quad\quad 3\bdot 3\bdot 3\bdot 6\bdot 6  \quad\quad 5\bdot 5\bdot 5\bdot 2\bdot 2  \quad\quad 7\bdot 7\bdot 7\bdot 6\bdot 6.
  \end{aligned}
\]

Hence, the hypergraph $\cH_5(C_8)$ has hyperedges as follows:
\begin{figure}[!h]
\centering
\begin{tabular}{ c||c|c|c|c|c|c|c|c|c|c }
Hyperedge & $\{1\}$ & $\{3\}$ & $\{5\}$ & $\{7\}$ & $\{1,2\}$ & $\{1,3\}$ & $\{3,6\}$ & $\{5,2\}$ & $\{5,7\}$ & $\{7,6\}$ \\[.5em] \hline
\rule{0pt}{13pt} Weight & 1 & 1 & 1 & 1 & 2 & 2 & 2 & 2 & 2 & 2
\end{tabular}
\end{figure}

As an example, we estimate $\EE_{2,1}$ by considering $\EE_{2,\{a\}}$ for $a \in C_8$. First, let $a=1$. Note that our goal here is not to get the exact value of $\EE_{2,\{1\}}$ but to obtain a uniform upper bound of $\EE_{2,\{a\}}$ for all $a\in C_8$.
For a generator $g$, there are several cases we need to deal with:

$\bullet$ \textbf{Case 1 ($a=1=g$):} Trivially, $g=1$. Since $(x'_1, x'_2)=(0,1), (0,2),$ or $(1,1)$, we have $(x_1, x_2)=(1,2), (1,3),$ or $(2,2)$, and hence, all zero-sum free sequences of this case in $C_8$ are
\[
  \begin{aligned}
  1\bdot 1\bdot 1\bdot 1\bdot (2\cdot 1) \quad \quad 1\bdot 1\bdot 1\bdot 1\bdot (3\cdot 1) \quad \quad 1\bdot 1\bdot 1\bdot (2\cdot 1)\bdot (2\cdot 1).
  \end{aligned}
  \]
  Thus, the expected number of all zero-sum free sequences of this case in $R$ is $$X_{2,1}\cdot p.$$

$\bullet$ \textbf{Case 2 ($a=1=2g$):} There is no such $g$, but we go forward to get a uniform upper bound. Since $a=2g$, we have $x'_\ell=1$ for some $\ell$. Hence, the number of all zero-sum free sequences of this case in $C_8$ is
  at most $X_{2-1,1}+X_{2-1,0}$, where the first term is from the situation when all other $x'$ are different from $x'_\ell$ and the second term is from the other situation.
  Thus, the expected number of all zero-sum free sequences of this case in $R$ is at most
  $$\(X_{2-1,1}+X_{2-1,0}\) p.$$

$\bullet$ \textbf{Case 3 ($a=1=3g$):} We infer that $g=3$ and $(x'_1, x'_2)=(0,2)$. Hence, every zero-sum free sequences of this case in $C_8$ is
  \[
  \begin{aligned}
  3\bdot 3\bdot 3\bdot 3\bdot (3\cdot 3)=3\bdot 3\bdot 3\bdot 3\bdot 1.
  \end{aligned}
  \]
The number of all zero-sum free sequences of this case over $C_8$ is $$X_{2-2,1}+X_{2-2,0} \leq X_{2-1,1}+X_{2-1,0}.$$
Thus, the expected number of all zero-sum free sequences of this case in $R$ is at most $$\(X_{2-1,1}+X_{2-1,0}\) p.$$

Therefore, $$\EE_{2,\{1\}}\leq \(X_{2,1}+2\(X_{2-1,1}+X_{2-1,0}\)\)p.$$
By the same argument, for every $a\in C_8$, we have that $\EE_{2,\{a\}}$ has the same upper bound, and hence,
 $$\EE_{2,1}\leq \(X_{2,1}+2\(X_{2-1,1}+X_{2-1,0}\)\)p.$$
In a similar way, one can estimate $\EE_2$ in  $C_8$, which gives $\EE'$ and $\EE^*$.

\subsection{Estimating $\EE_{d,i}$}$\phantom{}$
We are ready to estimate $\EE_{d,i}$ in a general $C_n$. First, we consider the case where $i=1$.

\begin{lemma}\label{lem:i=1}
For $2\leq d\leq \left\lfloor \frac{1+\sqrt{1+8k}}{2} \right\rfloor$, we have that
\begin{eqnarray*}\EE_{d,1}&\leq& p^{d-1}\(X_{k,d-1}+k(X_{k-1,d-1}+X_{k-1,d-2})\)\mbox{ and }\\
\EE_{1,1}&=&1.
\end{eqnarray*}
\end{lemma}

\begin{proof}
Fix $a\in C_n$. We estimate the expected number of zero-sum free sequences $$g \bdot \ldots \bdot g \bdot (x_1 g) \bdot \ldots \bdot (x_k g)$$ in $R$ containing $\{a\}$ with two cases separately: for a generator $g$, the first case is when $a=g$, and the second case is when $a=jg$ for $2\leq j\leq k+1$.

$\bullet$ \textbf{Case 1 ($a=g$):} The number of zero-sum free sequences over $C_n$ containing $a=g$ is
$ X_{k,d-1}.$ Hence, the expected number of zero-sum free sequences in $R$ containing $a=g$ is
\begin{equation}\label{eq:i=1(1)} X_{k,d-1}\cdot p^{d-1}.\end{equation}

$\bullet$ \textbf{Case 2 ($a=jg$ for $2\leq j\leq k+1$):}
We first estimate the number of zero-sum free sequences over $C_n$ containing $a=jg=x_\ell g$ for some $\ell$. Since $x'_\ell=x_\ell-1\geq 1$, the remaining $x'_1,\dots,x'_{\ell-1}, x'_{\ell+1}, \dots, x'_k$ satisfy $\sum_{\overset{1\leq i\leq k}{i\neq \ell}} x'_i\leq k-1$. If $x'_1,\dots,x'_{\ell-1}, x'_{\ell+1}, \dots, x'_k$ are different from $x'_\ell$, then the number of zero-sum free sequences over $C_n$ is at most
$ X_{k-1,d-2}.$ Otherwise, the number of zero-sum free sequences over $C_n$ is at most $X_{k-1,d-1}.$
Since $2\leq j\leq k+1$, the expected number of zero-sum free sequences of this case in $R$ is
\begin{equation}\label{eq:i=1(2)} k(X_{k-1,d-1}+X_{k-1,d-2}) p^{d-1}.\end{equation}

From~\eqref{eq:i=1(1)} and~\eqref{eq:i=1(2)}, we have that
\begin{equation*}\EE_{d,1}\leq \max_{\{a\}}\EE_{d,\{a\}}\leq p^{d-1}\(X_{k,d-1}+k(X_{k-1,d-1}+X_{k-1,d-2})\),
\end{equation*}
which completes our proof of the lemma.
\end{proof}

Next, we consider a general $i$ with $|A|=i$.

\begin{lemma}\label{lem:general_i}
For $1\leq i< d\leq \left\lfloor \frac{1+\sqrt{1+8k}}{2} \right\rfloor$, we have that
\begin{eqnarray*}\EE_{d,i}&\leq&  p^{d-i} \left[ i {k\choose i-1}+{k\choose i}\right]\(\sum_{j=0}^{i}X_{k-i+1,d-1-j}\) \mbox{ and }\\
\EE_{d,d}&=&1.
\end{eqnarray*}
\end{lemma}

\begin{proof}
Fix $a_1,a_2, \dots, a_i\in C_n$. We estimate the number of zero-sum free sequences $$g \bdot \ldots \bdot g \bdot (x_1 g) \bdot \ldots \bdot (x_k g)$$ containing $\{a_1,a_2,\dots, a_i\}$ with two cases separately: for a generator $g$, the first case is when $g=a_\ell$ for some $\ell$, and the second case is when $g\neq a_\ell$ for all $\ell$.

$\bullet$ \textbf{Case 1 ($a_1=g$ and $\{a_2,\dots, a_i\}=\{j_2g,\dots, j_i g\}$ for $2\leq j_2<\dots<j_i\leq k+1$):} For fixed $g$ and $j_2,\dots, j_i$, the number of zero-sum free sequences over $C_n$ containing $\{a_1, a_2,\dots, a_i\}=\{g, j_2g,\dots, j_i g\}$
 is at most
$ X_{k-i+1,d-1}+\dots+X_{k-i+1,d-i}.$
The number of choices $(g,j_2,\dots,j_i)$ such that $g=a_\ell$ for some $\ell$ and $2\leq j_2<\dots<j_i\leq k+1$ is at most
$ i {k\choose i-1}.$  Hence, the expected number of zero-sum free sequences in $R$ containing $\{a_1, a_2,\dots, a_i\}=\{g, j_2g,\dots, j_i g\}$ is
\begin{equation}\label{eq:general_i(1)}i {k\choose i-1}(X_{k-i+1,d-1}+\dots+X_{k-i+1,d-i}) p^{d-i}.\end{equation}

$\bullet$ \textbf{Case 2 ($\{a_1, a_2,\dots, a_i\}=\{j_1g, j_2g,\dots, j_i g\}$ for $2\leq j_1< \dots <j_i\leq k+1$):}
For fixed $j_1<\dots<j_i$, we first consider the number of zero-sum free sequences over $C_n$ containing $\{a_1, a_2,\dots, a_i\}=\{j_1g, j_2g,\dots, j_i g\}$. Without loss of generality, we let $x_1=j_1,\dots, x_i=j_i$.
Since $x'_\ell=x_\ell-1\geq 1$, the remaining $x'_{i+1},\dots, x'_k$ satisfy $\sum_{i+1\leq \ell\leq k} x'_\ell\leq k-i$. The number of distinct $x'_{i+1},\dots x'_k$ from $x'_1,\dots, x'_i$ are possibly $d-1$, $d-2$,\dots, or $d-1-i$, and hence, the number of zero-sum free sequences over $C_n$ containing $\{a_1, a_2,\dots, a_i\}=\{j_1g, j_2g,\dots, j_i g\}$ is at most
$ X_{k-i,d-1}+X_{k-i,d-2}+\dots+X_{k-i,d-1-i}.$ 
From the choices of $2\leq j_1<\dots<j_i\leq k+1$,  the expected number of zero-sum free sequences in $R$ containing $\{a_1,\dots, a_i\}$ is at most
\begin{equation}\label{eq:general_i(2)}{k\choose i} (X_{k-i,d-1}+X_{k-i,d-2}+\dots+X_{k-i,d-1-i})p^{d-i}.\end{equation}

From~\eqref{eq:general_i(1)} and~\eqref{eq:general_i(2)}, we have that
\begin{eqnarray*}\EE_{d,i}&\leq& \max_{\{a_1,\dots,a_i\}}\EE_{d,\{a_1,\dots,a_i\}}\\ &\leq &
\left[ i {k\choose i-1}+{k\choose i}\right]\(\sum_{j=0}^{i}X_{k-i+1,d-1-j}\) p^{d-i} ,
\end{eqnarray*}
which completes our proof of the lemma.
\end{proof}


\subsection{Proofs of Theorems~\ref{thm:concentration2} and~\ref{thm:concentration}}

\begin{proof}[Proof of Theorem~\ref{thm:concentration2}]

Let $X=N^R_{n-1-k}$. Under the assumption that $k$ is fixed and $p\ll 1$, Corollary~\ref{coro:exp2} gives that $$\EE(X)=p\varphi(n)\(1+O_k(p)\).$$
Since $k$ is fixed, Lemmas~\ref{lem:i=1} and~\ref{lem:general_i} yield that
\begin{eqnarray*}
\EE_1&=&\EE_{1,1}+\EE_{2,1}+\dots+\EE_{D,1}=O_k(1), \\
\EE_2&=&\EE_{2,2}+\EE_{3,2}+\dots+\EE_{D,2}=O_k(1), \\
&\vdots&\\
\EE_D&=&\EE_{D,D}=1.
\end{eqnarray*}

Hence,
\begin{eqnarray*}\EE'&=&\max_{1\leq i\leq D}\{\EE_i\}=O_k(1) \hskip 0.5em
\mbox{and} \\
\EE^*&=&\max\{\EE',\EE\}=p\varphi(n),\end{eqnarray*}
provided that $p\varphi(n)\gg 1$, i.e.,  $p\gg \frac{\log\log n}{n}.$

Set $\lambda=d\log n$, then $e^{-\lambda}n^{d-1}=1/n=o(1),$ and hence, the Kim--Vu polynomial concentration result (Theorem~\ref{thm:Kim--Vu}) gives that a.a.s.
$$|X-\EE(X)|=O_k\(\sqrt{p\varphi(n)}(\log n)^d\),$$
that is,
$$X=p\varphi(n)+O_k\(p^2\varphi(n)+\sqrt{p\varphi(n)}(\log n)^d\). $$
Note that $p\varphi(n)\gg \sqrt{
p\varphi(n)}(\log n)^d$
is equivalent to $ p\gg \frac{(\log n)^{2d}\log\log n}{n},$ and hence, our assumption on $p$ is
$$\frac{(\log n)^{2d}\log\log n}{n}\ll p\ll 1.$$
Thus, we complete the proof of Theorem~\ref{thm:concentration2}.
\end{proof}


For the proof of Theorem~\ref{thm:concentration}, we use the following version of Chernoff's bound.

\begin{lemma}[\textbf{Chernoff's bound}, Corollary 4.6 in~\cite{MU2005}] \label{lem:Chernoff} Let $X_i$ be independent random variables such that $$\mbox{$\Pr[X_i=1]=p_i$\;\; and \;\; $\Pr[X_i=0]=1-p_i$,}$$ and let $X=\sum_{i=1}^{n} X_i$.
For $0<\lambda<1$,
\begin{equation*}\Pr \Big[|X-\EE(X)|\geq \lambda\EE(X)\Big]\leq 2 \exp\Big(-\frac{\lambda^2}{3}\EE(X)\Big). \end{equation*}
\end{lemma}

We are ready to prove Theorem~\ref{thm:concentration}.

\begin{proof}[Proof of Theorem~\ref{thm:concentration}]
Let $X=N^R_{n-1-k,d}$. First, we consider the case where $d=1$. Observe that
$X\sim Bin(\varphi(n),p)$, and hence, Chernoff's bound with $\lambda=\omega/\sqrt{p\varphi(n)}$ implies that a.a.s.
$$|X-p\varphi(n)|<\omega(p\varphi(n))^{1/2},$$
provided that $p\varphi(n)\gg 1$, i.e., $p\gg \frac{\log\log n}{n},$ where $\omega$ tends to $\infty$ arbitrarily slowly as $n\rightarrow \infty$.

Next we consider the case when $d\geq 2$.
It follows from~\eqref{lem:expectation_N^R} that $$\EE(X)=p^d\varphi(n)X_{k,d-1}.$$
Lemma~\ref{lem:general_i} gives that for a fixed $k$,
\begin{eqnarray*}
\EE_{d,i}&=& O_{k}(p^{d-i})\hskip 0.5em \mbox{for $1\leq i\leq d-1$} \hskip 0.5em \mbox{and}\\
\EE_{d,d}&=&1.
\end{eqnarray*}
Hence,
\begin{eqnarray*}\EE'&=&\max_{1\leq i\leq d} \EE_{d,i}=O_{k}(1) \hskip 0.5em
\mbox{and} \\
\EE^*&=&O_{k}(\max\{1,p^d\varphi(n)\})=O_{k}(p^d\varphi(n)),
\end{eqnarray*}
provided that $p^d\varphi(n)\gg 1$, i.e., $p\gg \(\frac{\log\log n}{n}\)^{1/d}.$

Set $\lambda=d\log n$, then $e^{-\lambda}n^{d-1}=1/n=o(1),$ and hence,
the Kim--Vu polynomial concentration result (Theorem~\ref{thm:Kim--Vu}) implies that a.a.s.
\begin{equation*}
|X-\EE(X)|\leq a_d(\EE'\EE^*)^{1/2}\lambda^d=O_k\(p^{d/2}\varphi(n)^{1/2}(\log n)^d\).
\end{equation*}
Note that $p^d\varphi(n)\gg \sqrt{
p^d\varphi(n)}(\log n)^d$
is equivalent to $ p\gg \frac{(\log n)^2(\log\log n)^{1/d}}{n^{1/d}},$ which is our assumption on $p$.
This completes our proof of Theorem~\ref{thm:concentration}.
\end{proof}


\begin{acknowledge}
 The authors thank Myungho Kim (Kyung Hee Univ.) for valuable comments.  
\end{acknowledge}


\providecommand{\bysame}{\leavevmode\hbox to3em{\hrulefill}\thinspace}
\providecommand{\MR}{\relax\ifhmode\unskip\space\fi MR }
\providecommand{\MRhref}[2]{%
  \href{http://www.ams.org/mathscinet-getitem?mr=#1}{#2}
}
\providecommand{\href}[2]{#2}
\def\MR#1{\relax}


\endgroup

\end{document}